\documentclass[11pt,a4paper]{amsart}
\usepackage{amssymb,amsfonts,amsmath,mathrsfs}
\usepackage{comment}

\newtheorem{prop}{Proposition}[section]
\newtheorem{theorem}[prop]{Theorem}
\newtheorem{cor}[prop]{Corollary}
\newtheorem{lemma}[prop]{Lemma}

\theoremstyle{definition}
\newtheorem{Def}[prop]{Definition}

\newtheorem{rem}[prop]{Remark}
\newtheorem{example}[prop]{Example}

\numberwithin{equation}{section}

\def\R{\Bbb R}

\def\Dx{\Delta_x}
\def\Nx{\nabla_x}
\def\Dt{\partial_t}
\def\({\left(}
\def\){\right)}
\def\eb{\varepsilon}

\def\Cal{\mathcal}
\def\l{\lambda}

\def\l2l2{\Cal{L}(L^2(\Omega))}

\def\Bbb{\mathbb}
\def\<{\left<}
\def\>{\right>}
\def\divv{\operatorname{div}}
\def\Cal{\mathcal}
\begin{document}
\title[Kwak transform and inertial manifolds]{Kwak Transform and Inertial Manifolds revisited}
\author[A. Kostianko and S. Zelik] {Anna Kostianko${}^{1,2}$ and Sergey Zelik${}^{1,2}$}
\begin{abstract}The paper gives sharp spectral gap conditions for existence of inertial manifolds for abstract semilinear parabolic equations with non-self-adjoint leading part. Main attention is paid to the case where this leading part have Jordan cells which appear after applying the so-called Kwak transform to various important equations such as 2D Navier-Stokes equations, reaction-diffusion-advection systems, etc. The different forms of Kwak transforms and relations between them are also discussed.
\end{abstract}


\thanks{
 This work is partially supported by  the grant 19-71-30004 of RSF as well as  the EPSRC grant EP/P024920/1. The authors  would also like to thank A. Romanov   for stimulating discussions.}

\email{aNNa.kostianko@surrey.ac.uk}
\address{${}^1$ University of Surrey, Department of Mathematics, Guildford, GU2 7XH, United Kingdom.}
 \email{s.zelik@surrey.ac.uk}
\address{${}^2$ \phantom{e}School of Mathematics and Statistics, Lanzhou University, Lanzhou  \\ 730000,
P.R. China}
\maketitle
\section{Intorduction}\label{s0}
It is believed that the long-time dynamics generated by a dissipative PDE is effectively finite-dimensional, i.e., despite the infinite-dimensionality of the initial phase space, it can be governed by finitely many parameters (the so-called order parameters in the terminology of I. Prigogine) and the associated system of ODEs (the so-called inertial form (IF)) which describes the evolution of these order parameters.
\par
The problem of justification for such a reduction has been intensively studied during the last 30 years, see \cite{6,15,Mora,21,7,8, 18,Rom,9,10} and references therein. However, the precise mathematical meaning for this reduction remains a mystery. Indeed, the most popular construction for the above finite-dimensional reduction is based on constructing the so-called {\it global attractor} which is by definition a compact invariant set in the phase space which attracts the images of any bounded sets when time tends to infinity. The key result here is that under weak assumptions on the system considered, the global attractor exists and has finite Hausdorff and fractal dimensions. Together with the Man\'e projection theorem this give the desired finite-dimensional reduction as well as the IF, see \cite{6,24,25,8,10,19}.
\par
However, the above described scheme suffers from several essential drawbacks and hardly be considered as a reasonable solution of the above reduction problem. Namely, the IF constructed in this way is only H\"older continuous (which is not enough even for the uniqueness) and in general it is impossible to improve the regularity of the reduction. Another problem is that the involved Man\'e theorem gives a projection to a generic plane only and does not give any way to construct this plane explicitly. In addition, there are many examples appeared recently (see \cite{5,AZ1,AZ2,19}) which  show that dissipative systems generated even by parabolic equations in bounded domains may demonstrate features which cannot be interpreted as "finite-dimensional" in any reasonable sense. For instance, limit cycles with super-exponential rate of attraction, traveling waves in Fourier space, examples where the Hausdorff and fractal dimensions of the attractor are very different and depend on the choice of the phase space, etc. Thus, the finite-dimensional reduction problem occurs much more different and interesting than expected and requires further study.
\par
On the other hand, there is an ideal situation where the finite-dimension reduction works perfectly. Namely, when the considered system possesses the so-called inertial manifold (IM). By definition, this is at least Lipschitz (usually $C^{1+\eb}$) invariant finite-dimensional submanifold of the phase space with exponential tracking property (which usually a straightforward corollary of normal hyperbolicity). If such an object exists then we get the desired IF just by restricting our equations to the manifold, see \cite{15,17,29,8,18,19}. However, in contrast to global attractors, the existence of an IM requires strong restrictions (the so-called spectral gap conditions) which are not satisfied for many interesting equations including 2D Navier-Stokes problem. For instance, let us consider an abstract semilinear parabolic equation in a Hilbert space $H$:
\begin{equation}\label{0.abs}
\Dt u+Au=\Phi(u),\ \ u(t)\in H,\ t\ge0, \ u\big|_{t=0}=u_0,
\end{equation}
where $A:D(A)\to H$ is a positive {\it self-adjoint} linear operator with compact inverse and $\Phi$ is a nonlinearity which is, in a sense, subordinated to $A$. Namely, let $\{\lambda_k\}_{k=1}^\infty$ be the eigenvalues of the operator $A$ and $\{e_k\}_{k=1}^\infty$ be the corresponding eigenvectors. We denote by $H^s:=D(A^{s/2})$ the scale of Hilbert spaces generated by the operator $A$. Assume also that the map $\Phi$ is Lipschitz continuous as a map from $H$ to $H^\beta$ for some $\beta\in[0,-2)$, i.e.,
\begin{equation}\label{0.lip}
\|\Phi(u_1)-\Phi(u_2)\|_{H^\beta}\le L\|u_1-u_2\|_H,\ \ u_1,u_2\in H.
\end{equation}
Then the classical spectral gap condition reads: if there exists $n\in\Bbb N$ such that
\begin{equation}\label{0.gap}
\frac{\lambda_{n+1}-\lambda_n}{\lambda_n^{-\beta/2}+\lambda_{n+1}^{-\beta/2}}>L,
\end{equation}
then equation \eqref{0.abs} possesses an IM over the base generated by the linear combinations of the first $n$-eigenvectors. This also true for the case $\beta>0$ up to some minor changes, see \cite{19}. The most important for us are two cases: the case $\beta=0$ which corresponds to, say, reaction-diffusion equations where the spectral gap conditions read:
\begin{equation}\label{0.0gap}
\lambda_{n+1}-\lambda_n>2L
\end{equation}
and the case $\beta=-1$ which corresponds to, say, reaction-diffusion-advection equations or Navier-Stokes system where we need
\begin{equation}\label{0.1gap}
\sqrt{\lambda_{n+1}}-\sqrt{\lambda_n}>L
\end{equation}
to be satisfied. Keeping in mind the Weyl asymptotic for the eigenvalues of the Laplacian ($\lambda_m\sim Cm^{2/d}$) we see that, for the case of reaction-diffusion equations, the spectral gap conditions are automatically satisfied (for the properly chosen $n$) in the 1D case only and become problematic already in~2D. For the case of reaction-diffusion-advection equations, the situation is much worse since the spectral gap conditions fail already in 1D.
\par
It is also known that the above stated spectral gap conditions are sharp in the class of abstract semilinear parabolic equations in the sense that if they are violated it is possible to construct an equation from this class which does not possess an IM (see \cite{Mora,22,19} and references therein). However, the situation may become better if more concrete classes of equations are considered (e.g., a dissipative system which is generated by a PDE, which does not contain any pseudo-differential or non-local operators). In such classes IMs may exist even if the spectral gap condition is violated for all~$n$.
\par
Actually, up to the moment there are two different methods to get the existence of an IM beyond the spectral gap conditions. The first one is the so-called spatial averaging method suggested in \cite{16} for the case of 2D or 3D {\it scalar} reaction-diffusion equations with periodic boundary conditions, see also \cite{AZ3,A} for extensions to the case of Cahn-Hilliard equations, modified Navier-Stokes equations, etc. The key drawback is that this method usually works only for scalar equations and only for periodic boundary conditions (or close to that conditions, see \cite{Kwen}).
\par
An alternative method, which is potentially more promising but essentially less understood, is based on the idea to transform the initial equation to a new form or/and  to embed it to a larger system of equations in such a way that the new equations will satisfy spectral gap conditions. The most essential recent progress achieved by using this method is clarifying the situation with IMs for 1D reaction-diffusion-advection systems, see \cite{AZ1,AZ2} and also Section \ref{s1} below.
\par
To the best of our knowledge, the idea to use such embeddings/transforms in the theory of IMs follows from Kwak \cite{Kwak}. In this work a special transform which reduces 2D Navier-Stokes system on a torus to a larger system
\begin{equation}\label{0.kwak}
\Dt\(\begin{matrix}u\\v\end{matrix}\)+\(\begin{matrix}1&1\\0&1\end{matrix}\)A
\(\begin{matrix}u\\v\end{matrix}\)=\Bbb F(u,v),
\end{equation}
in a product $\Bbb H=H\times H$ of Hilbert spaces and with non-linearity $\Bbb F$ satisfying \eqref{0.lip} with $\beta=0$ has been constructed. Being precise, the initial Kwak transform gives slightly more complicated than \eqref{0.kwak} equations, but using the modification suggested in \cite{Rom}, one can get equations \eqref{0.kwak} even with simpler nonlinearity:
\begin{equation}\label{0.kwa-kwak}
\Dt\(\begin{matrix}u\\v\end{matrix}\)+\(\begin{matrix}1&1\\0&1\end{matrix}\)A
\(\begin{matrix}u\\v\end{matrix}\)=\(\begin{matrix}0\\F(u)\end{matrix}\),
\end{equation}
where the new non-linearity $F$ satisfies \eqref{0.lip} with $\beta=0$, see Section \ref{s1} below for more details.
\par
Unfortunately, the original paper of Kwak \cite{Kwak} (as well as the works of his successors, see \cite{Kwak1,Tem,Rom}) contains a {\it crucial} error related with an implicit assumption that the spectral gap conditions for equation \eqref{0.kwak} with {\it non-self-adjoint} leading part are the same as for equation \eqref{0.abs} with {\it self-adjoint} leading part. Namely, the authors apply conditions of the form \eqref{0.0gap} in order to construct an IM for equation \eqref{0.kwak}. In reality the spectral gap conditions for equations \eqref{0.kwak} and \eqref{0.kwa-kwak} differ drastically from the ones for the self-adjoint case and are more close to \eqref{0.1gap} rather than to \eqref{0.0gap} (due to the presence of Jordan cells in the leading part). Thus, the original Kwak's approach fails and the problem for the existence of an IM for 2D Navier-Stokes equations remains open. In addition, a counterexample of 1D reaction-diffusion-advection system with periodic boundary conditions which does not possess any IM has been recently constructed in \cite{AZ2}. This class of equations possesses a Kwak-type transform and can be reduced to the form \eqref{0.kwa-kwak}, see Section \ref{s1}. This confirms from the other side that the Kwak transform is not sufficient to construct an IM for such class of equations.
\par
Although the presence of an error in the above mentioned works is known for a long time, it is surprisingly difficult to find in the literature even the precise explanation where exactly the error is. Moreover, again to the best of our knowledge, the precise spectral gap conditions for equations \eqref{0.kwak} and \eqref{0.kwa-kwak} have been not known before. The main aim of the present paper is to cover this gap and to give sharp spectral gap conditions for both equations \eqref{0.kwak} and \eqref{0.kwa-kwak}. The next theorem, proved in Section \ref{s3}, can be treated as a main result of the paper.

\begin{theorem}\label{Th0.main} Let $A:D(A)\to H$ be a linear self-adjoint positive operator in a Hilbert space $H$ with compact inverse and let $\Bbb H:=H\times H$. Let also the nonlinearity $\Bbb F:\Bbb H\to\Bbb H$ be globally Lipschitz continuous with the Lipschitz constant $L$. Assume that there exists $n\in\Bbb N$ such that
\begin{equation}\label{0.k-gap}
\frac{(\lambda_{n+1}-\lambda_n)^2}
{\lambda_{n+1}+\lambda_n+2\sqrt{\lambda_n^2-\lambda_n\lambda_{n+1}+\lambda_{n+1}^2}}>L.
\end{equation}
Then, equation \eqref{0.kwak} possesses $2n$-dimensional IM with the base $P_n\Bbb H:=P_nH\times P_nH$ (here and below $P_n$ stands for the orthoprojector to the first $n$ eigenvectors of the operator $A$). Moreover, the spectral gap assumption is sharp in the following sense: if \eqref{0.k-gap} is violated there are examples of equations in the form \eqref{0.kwak} where such an IM does not exist.
\end{theorem}
The spectral gap conditions become essentially simpler if the non-linearity $\Bbb F$ has a special form of \eqref{0.kwa-kwak}.
\begin{theorem}\label{Th0.main1} The analogue of Theorem \ref{Th0.main} holds for equation \eqref{0.kwa-kwak} with the improved spectral gap conditions which in this case read: there exists $n\in\Bbb N$ such that
\begin{equation}\label{0.kkgap}
\sqrt{\lambda_{n+1}}-\sqrt{\lambda_n}>\sqrt{L}.
\end{equation}
\end{theorem}
We see that this spectral gap condition almost coincides with the condition \eqref{0.1gap} which we have initially before applying the Kwak transform and this clarifies why the Kwak transform in its original form does not help to construct an IM. The presence of the "mysterious" term $\sqrt L$ is explained in Remark \ref{Rem3.last}. Nevertheless, as we have already mentioned, more advanced embeddings/transforms (which, say, do not destroy the self-adjoint structure of the leading part) may be effective tools for establishing  the existence of IMs and definitely deserve further attention.
\par
The paper is organized as follows. In section \ref{s1} we give a number of examples of Kwak-type transforms, some of them are well-known, but others look new, and discuss the relations between them. In particular, we present here a bit unexpected connection between Kwak transform and wave equations with structural damping, see Remark \ref{Rem1.wave} below.
\par
Section \ref{s2} is devoted to study the linear problem of the form \eqref{0.kwak} in the properly chosen weighted spaces of trajectories. This is the central part of the paper and the estimates obtained there are crucial for our construction of an IM.
\par
Finally, Theorems \ref{Th0.main} and \ref{Th0.main1} are proved in Section \ref{s3} using the so-called Perron method. Our approach to find sharp spectral gap conditions is inspired by \cite{17} and also by more recent work \cite{ACZ}.

\section{Examples of Kwak-type transforms}\label{s1}
In this section, we consider several examples related with Kwak transforms for various classes of semilinear parabolic problems. We start with the simplest Burger's equation with periodic boundary conditions.
\begin{example}\label{Ex1.1} Let us consider the  viscous Burger's equation
\begin{equation}\label{1.bur}
\Dt u =\nu\partial_x^2 u+\partial_x(u^2)-f(u),\ \ x\in(-\pi,\pi)
\end{equation}
endowed with periodic BC. Here $\nu>0$ is a given parameter and $f(u)$ is a given smooth nonlinearity. The full nonlinearity here is $F(u):=\partial_x(u^2)-f(u)$. Since it contains $\partial_xu$, it decreases the smoothness by one, namely, $F$ is a smooth map from $H^s_{per}(-\pi,\pi)$ to $H^{s-1}_{per}$ if $s$ is large enough (say, $s>\frac12$). Thus, we need to take $\beta=-1$ in the condition \eqref{0.gap}. Using that the eigenvalues are $\lambda_N=\nu N^2$ (with multiplicity two), the spectral gap condition reads
$$
\frac{\lambda_{N+1}-\lambda_N}{\lambda_{N+1}^{1/2}+\lambda_N^{1/2}}=\nu>L
$$
and do not hold if $L\ge\nu$.
\par
The key idea of Kwak was to embed equation \eqref{1.bur} to a larger system of semilinear equations in such a way that the new non-linearity will be more regular, for instance, will not contain the spatial derivatives $\partial_x u$ which, in turn, would allow us to use \eqref{0.gap} with $\beta=0$, see \cite{Kwak, Kwak1}. To realize this idea, we introduce the new variables $v(t):=\partial_x u$ and $w(t):=\nu^{-1}u^2(t)$. Then, differentiating \eqref{1.bur} in $x$, after straightforward computations, we arrive at
\begin{multline}
\Dt u=\nu\partial_x^2 u+2uv-f(u),\ \ \Dt v=\nu\partial_x^2 v+\nu\partial_x^2 w-f'(u)v, \\
\Dt w=\nu\partial_x^2 w-2v^2+2\nu^{-1}u(2uv-f(u)).
\end{multline}
Thus, the new system of semilinear equations reads:
\begin{multline}
\Dt\(\begin{matrix}u\\v\\w\end{matrix}\)-
\nu\(\begin{matrix}1&0&0\\0&1&1\\0&0&1\end{matrix}\)\partial_x^2
\(\begin{matrix}u\\v\\w\end{matrix}\)=
\(\begin{matrix}2uv-f(u)\\-f'(u)v\\2\nu^{-1}u(2uv-f(u))-2v^2\end{matrix}\)
\end{multline}
and we see that the new nonlinearity  indeed does not contain the spatial derivatives and acts from $H^s_{per}$ to $H^s_{per}$, so $\beta=0$.
\end{example}
\begin{example}\label{Ex1.2}
 Let us consider the following reaction-diffusion system in a bounded domain $\Omega\subset\R^2$,  say, with Dirichlet boundary condition:
\begin{equation}\label{1.rds}
\Dt u-\Dx u=f(u),\ \ u\big|_{\partial\Omega}=0,\ \ f(0)=0.
\end{equation}
Here $\beta=0$, but since $\lambda_N\sim CN$ due to the Weyl asymptotic, this is still not enough for the spectral gap condition to be satisfied. Following to Romanov (see \cite{Rom}), we introduce the variable $v(t):=(\Dx)^{-1}f(u(t))$. Then, this function solves
\begin{multline*}
\Dt v-\Dx v=(\Dx)^{-1}[f'(u)(\Dt u-\Dx u)-f''(u)|\Nx u|^2]=\\=(\Dx)^{-1}\(f'(u)f(u)-f''(u)|\Nx u|^2\):=F(u).
\end{multline*}
Then, using the Sobolev embedding theorem, it is not difficult to verify that $F$ is a smooth map from $H^{2+\eb}_{\Dx}$ to $H^{3+\eb}_{\Dx}$ for any $0<\eb<1/2$ (here we denote $H^s_\Delta:=D((-\Dx)^{s/2}$)) and therefore, we have reduced the initial problem to
\begin{equation}
\Dt \(\begin{matrix}u\\v\end{matrix}\)+\(\begin{matrix}1&1\\0&1\end{matrix}\)(-\Dx)
\(\begin{matrix}u\\v\end{matrix}\)=\(\begin{matrix}0\\F(u)\end{matrix}\),
\end{equation}
where the new nonlinearity satisfies the Lipschitz assumption with $\beta=1$.
\end{example}
\begin{example}\label{Ex1.3} Let us  consider the 1D reaction-diffusion-advection system
\begin{equation}\label{1.rda}
\Dt u+(1-\partial_x^2)u =f(u,u_x),\  x\in(-\pi,\pi)
\end{equation}
endowed with periodic boundary conditions. Here $u=u(t,x)=(u^1,\cdots,u^k)$ is an unknown vector-valued function and $f$ is a given smooth non-linearity.
\par
Following Example \ref{Ex1.2}, we introduce the new variable
$$
v(t):=(\partial_x^2-1)^{-1}f(u(t),u_x(t))
$$
 which solves
\begin{multline}\label{1.F}
\Dt v+(1-\partial_x^2)v=(\partial^2_x-1)^{-1}(\Dt f+(1-\partial_x^2)f)=\\=
(\partial^2_x-1)^{-1}\( f'_u(\Dt u+(1-\partial_x^2)u)+f'_{u_x}(\Dt u_x+\right.\\+\left.(1-\partial_x^2)u_x)-f'_uu-f'_{u_x}u_x-f+\right.\\\left.+f''_{u,u}[u_x,u_x]+2f''_{u,u_x}[u_x,u_{xx}]+f''_{u_x,u_x}[u_{xx},u_{xx}]\)=\\=
(\partial^2_x-1)^{-1}\(f'_uf+f'_{u_x}(f'_uu_x+f'_{u_x}u_{xx})-f'_uu-f'_{u_x}u_x-f+\right.\\
\left.+f''_{u,u}[u_x,u_x]+2f''_{u,u_x}[u_x,u_{xx}]+f''_{u_x,u_x}[u_{xx},u_{xx}]\):=F(u).
\end{multline}
Using the elliptic regularity and the fact that $H^s_{per}\subset C$ if $s>\frac12$, we see that the map $F$ is well-defined and smooth as a map from $H^{s}_{per}(-\pi,\pi)$ to $H^s_{per}(-\pi,\pi)$ for $s>\frac52$. Thus, the initial reaction-diffusion-advection problem is reduced to the following one
\begin{equation}\label{1.jordan}
\Dt \(\begin{matrix} u\\v\end{matrix}\)+\(\begin{matrix}1&1\\0&1\end{matrix}\)A\(\begin{matrix} u\\v\end{matrix}\)=\(\begin{matrix} 0\\F(u)\end{matrix}\),
\end{equation}
where $A:=(1-\partial_x^2)$.
\end{example}
\begin{example}\label{Ex1.4} Consider 2D Navier-Stokes equation with periodic boundary conditions:
\begin{equation}
\Dt u-\Dx u=-(u,\Nx) u-\Nx p+g,\ \ \divv u=0,
\end{equation}
where $u=(u_1,u_2)$ is a velocity vector field, $p$ is pressure and $g$ are given smooth external forces. We assume that $g$ and $u$ have zero means and denote by $\Pi$ the standard Leray projector to divergent free vector fields. Recall also that in the case of periodic boundary conditions $\Pi$ commutes with the Laplacian. Analogously to Examples \ref{Ex1.1} and \ref{Ex1.2}, we introduce a new variable
$$
v(t):=(-\Dx)^{-1}\Pi[(u(t),\Nx)u(t)-g]=(-\Dx)^{-1}\Pi(u_1\partial_{x_1}u+u_2\partial_{x_2}u-g),
$$
where $(-\Dx)^{-1}$ is the inverse Laplacian with periodic boundary conditions and zero mean.
Then, this function satisfies
\begin{multline}
\Dt v-\Dx v=(-\Dx)^{-1}\Pi\((\Dt u-\Dx u,\Nx)u+\right.\\
\left.+(u,\Nx)(\Dt u-\Dx u)-2(\Nx u,\Nx)\Nx u\)-\Pi g=\\=(-\Dx)^{-1}\Pi\(-(\Pi(u,\Nx)u,\Nx)u-(u,\Nx)(\Pi(u,\Nx)u)+(\Pi g,\Nx)u+\right.\\\left.(u,\Nx)(\Pi g)-2(\Nx u,\Nx)\Nx u\)-\Pi g:=F(u)
\end{multline}
and using that $H^s_{per}\subset C$ if $s>\frac32$, together with the elliptic regularity for the Leray projector, we see that $F$ maps $H^s_{per}\cap\{\divv u=0\}$ to itself if $s>\frac52$, so, similarly to the reaction-diffusion-advection case, we again have $\beta=0$ for this non-linearity and the transformed equation has the form of \eqref{1.jordan} with $A=-\Dx$.
\end{example}
\begin{rem}\label{Rem1.wave} The transform presented in Example \ref{Ex1.4} differs slightly from the original Kwak transform for 2D Navier-Stokes equation, see \cite{Kwak} and \cite{9} and is inspired by the construction from \cite{Rom} discussed in Example \ref{Ex1.2}. The advantage of this modified Kwak transform is that the transformed system has less number of equation (in comparison with the original version) and also has more transparent structure, namely, the first component of the nonlinearity in \eqref{1.jordan} vanishes and the second one depends only on $u$. This structure allows us to reduce the transformed system \eqref{1.jordan} to the second order scalar equation by expressing $v$ through $u$ from the first equation and inserting the result to the second equation:
\begin{equation}\label{1.wave}
A^{-1}\Dt^2 u+2\Dt u+Au=-F(u),\ \ \text{or}\ \ \Dt^2 u+2A\Dt u+A^2 u=-AF(u).
\end{equation}
This in turn gives an interesting connection between Navier-Stokes and wave equations with structural damping realized via the Kwak transform.
\end{rem}
\begin{rem}\label{Rem1.bla} The example of a 1D system of reaction-diffusion-advection equations of the form \eqref{1.rda} which does not possess any finite-dimensional inertial manifold has been presented in \cite{AZ2}. Moreover, in this example two trajectories $u_1(t)$ and $u_2(t)$ of equation \eqref{1.rda} belonging to the global attractor $\mathcal A$ such that
\begin{equation}\label{1.e3}
\|u_1(t)-u_2(t)\|_{L^2}\le Ce^{-\alpha t^3},\ \ t\ge0, \ \ C,\alpha>0
\end{equation}
have been explicitly constructed. Thus, at least for the case of reaction-diffusion-advection problems, the possibility to make the Kwak transform and to reduce the system to the form \eqref{1.jordan} is still {\it not enough} to get an inertial manifold. In particular, we cannot use the standard spectral gap condition
$$
\lambda_{N+1}-\lambda_N>2L
$$
in the case where the leading operator is not self-adjoint and possesses Jordan cells. One more interesting  observation is that the considered reaction-diffusion-advection problem {\it cannot be embedded} into a larger system of semilinear parabolic equations of the form
\begin{equation}\label{1.sa}
\Dt U+\Bbb A U=\Bbb F(U)
\end{equation}
in a proper Hilbert space $\Bbb H$
with positive {\it self-adjoint} operator $\Bbb A$ with compact inverse and a Lipschitz non-linearity $\Bbb F:\Bbb H\to\Bbb H$. Indeed, as known, see e.g. \cite{19}, two bounded trajectories of \eqref{1.sa} cannot approach each other faster than $e^{-\alpha t^2}$, so the existence of such an embedding contradicts \eqref{1.e3}.
\par
On the other hand, equation \eqref{1.jordan} can be easily transformed back to the form \eqref{1.sa} with {\it self-adjoint} operator $\Bbb A$ and more singular nonlinearity $\Bbb F$. Indeed, introducing the variable $\tilde u(t):=A^{-1/2}u(t)$, we rewrite \eqref{1.jordan} as follows:
\begin{equation}\label{1.saa}
\Dt\(\begin{matrix}\tilde u\\v\end{matrix}\)+\(\begin{matrix}1&0\\0&1\end{matrix}\)A
\(\begin{matrix}\tilde u\\v\end{matrix}\)=\(\begin{matrix} -A^{1/2}v\\F\(A^{1/2}\tilde u\)\end{matrix}\):=\Bbb F(\tilde u,v),
\end{equation}
so if $F(u)$ acts from $H^s$ to $H^s$, the nonlinearity $\Bbb F$ will act from $\Bbb H^{s+1}$ to $\Bbb H^{s}$.
\end{rem}
\begin{example}\label{Ex1.5} Iterations of the Kwak transform. Let us return to Example \ref{Ex1.3} and introduce one more variable $w(t):=-A^{-1}F(u(t))$, where $F(u)$ is defined in \eqref{1.F}. Then, as elementary calculations show, the triple $(u,v,w)$ solves
\begin{equation}
\Dt \(\begin{matrix}u\\v\\w\end{matrix}\)+\(\begin{matrix}1&1&0\\0&1&1\\0&0&1\end{matrix}\)
A\(\begin{matrix}u\\v\\w\end{matrix}\)=\(\begin{matrix}0\\0\\\Phi(u)\end{matrix}\),
\end{equation}
where the non-linearity $\Phi$ acts from $H^s_{per}(-\pi,\pi)$ to $H^{s+1}_{per}(-\pi,\pi)$ for $s>\frac72$ and therefore $\beta=1$ for this transformed system. Since the initial equation may not have an inertial manifold, we see that the appearance of larger Jordan cells in the leading linear part requires stronger spectral gap assumptions for the inertial manifold to exist. The described scheme may be further iterated. This will lead to more and more regularizing   nonlinearities, but the advantage of this will be neglected by larger and larger Jordan cells in the leading part. Note also that the analogous transformations work for the Navier-Stokes system as well.
\end{example}

\section{Key estimates for the linear equation}\label{s2}
The aim of this section is to compute the norms of solution operators for the linearized equations associated with problem \eqref{0.kwak} in the corresponding weighted spaces. These estimates will be crucially used in the next section for constructing the inertial manifolds for the non-linear problem.
\par
We recall that
$H$ is a Hilbert space and $A: D(A)\to H$ is a linear (unbounded) positive self-adjoint operator with compact inverse. Let also $\{\lambda_n\}_{n=1}^\infty$ be its eigenvalues enumerated in the non-decreasing order and $\{e_n\}_{n=1}^\infty$ be the corresponding orthonormal base of eigenvectors. Finally, let $\Bbb  H:=H\times H$ and
$$
\Bbb A:=\(\begin{matrix} 1&1\\0&1\end{matrix}\)A.
$$
We consider the following linear non-homogeneous equation in the space $\Bbb H$:
\begin{equation}\label{2.1}
\partial_t\xi+\Bbb A\xi=h,\ \ \xi=(u,v)^t,\ \ h=(f,g)^t,
\end{equation}
where $t\in\R$ and the right-hand side $h$ belongs to the weighted space $L^2_{e^{\theta t}}(\R,\Bbb H)$ for some fixed exponent $\theta$. We also recall that this space is a subspace of $L^2_{loc}(\R,\Bbb H)$ defined by the following norm:
\begin{equation}
\|\xi\|^2_{L^2_{e^{\theta t}}(\R,\Bbb H)}:=\int_\R e^{2\theta t}\|\xi(t)\|^2_{\Bbb H}\,dt<\infty
\end{equation}
and $\|\xi\|^2_{\Bbb H}:=\|u\|^2_H+\|v\|^2_H$.
\par
 It is not difficult to see that, in the non-resonant case where
$$
\theta\ne\lambda_k,\ \ k\in\Bbb N,
$$
equation \eqref{2.1} is uniquely solvable in the space $L^2_{e^{\theta t}}(\R,\Bbb H)$, so the solution operator
$$
\Bbb L: L^2_{e^{\theta t}}(\R,\Bbb H)\to L^2_{e^{\theta t}}(\R,\Bbb H),\ \ \Bbb Lh:=\xi
$$
is well-defined.
\par
Our task now is to compute explicitly the norm of this operator and minimize it with respect to $\theta\in(\lambda_n,\lambda_{n+1})$. The answer is given by the following proposition.

\begin{prop}\label{Prop2.1} Let $\lambda_{n+1}>\lambda_n$ and $\theta\in(\lambda_n,\lambda_{n+1})$. Then the minimal value of the norm of the solution operator $\Bbb L$ is achieved for
\begin{equation}\label{2.min}
\theta=\frac23(\lambda_{n+1}+\lambda_n)-\frac13\sqrt{\lambda_{n+1}^2-\lambda_n\lambda_{n+1}+\lambda_n^2}
\end{equation}
and is equal to
\begin{equation}\label{2.15}
\|\Bbb L\|_{\Cal L(L^2_{e^{\theta t}}(\R,\Bbb H),L^2_{e^{\theta t}}(\R,\Bbb H))}=\frac{\lambda_{n+1}+\lambda_n+2\sqrt{\lambda_{n+1}^2-\lambda_n\lambda_{n+1}+\lambda_n^2}}
{(\lambda_{n+1}-\lambda_n)^2}.
\end{equation}
\end{prop}
\begin{proof}
 First, we make  change $\tilde\xi(t):=e^{\theta t}\xi(t)$ of the independent variable which reduces problem \eqref{2.1} to
\begin{equation}\label{2.2}
\partial_t\tilde \xi+(\Cal A-\theta)\tilde\xi=\tilde h,\ \ \tilde \xi:=(\tilde u,\tilde v)^t,\ \ h=(\tilde f,\tilde g)^t,
\end{equation}
where $\tilde \xi,\tilde h\in L^2(\R,\Bbb H)$. Thus,  estimating the solution $\xi$ of \eqref{2.1} in the weighted space $L^2_{e^{\theta t}}(\R,\Bbb H)$ is equivalent to estimating the solution $\tilde \xi$ of equation \eqref{2.2} in the non-weighted space $L^2(\R,\Bbb H)$.
\par
Second, we expand the solution $\tilde \xi(t)=\sum_{n=1}^\infty \tilde\xi_n(t) e_n$ where the functions $\tilde \xi_n(t)=(\tilde u_n(t),\tilde v_n(t))^t$ solve
\begin{equation}\label{2.3}
\partial_t\(\begin{matrix} \tilde u_n\\\tilde v_n\end{matrix}\)+\(\begin{matrix} \lambda_n-\theta&\lambda_n\\0&\lambda_n-\theta\end{matrix}\)\(\begin{matrix} \tilde u_n\\\tilde v_n\end{matrix}\)=\(\begin{matrix} \tilde f_n\\\tilde g_n\end{matrix}\).
\end{equation}
If we denote by $\Bbb L_n: [L^2(\R)]^2\to [L^2(\R)]^2$ the solution operators of problems \eqref{2.3}, then, due to the Parseval equality,
\begin{equation}\label{2.4}
\|\Bbb L\|_{\Cal L(L^2_{e^{\theta t}}(\R,\Bbb H),L^2_{e^{\theta t}}(\R,\Bbb H))}=\sup_{n\in\Bbb N}\|\Bbb L_n\|_{\Cal L([L^2(\R)]^2,[L^2(\R)]^2)}
\end{equation}
and we only need to find the norms of operators  $\Bbb L_n$. To this end, we do the Fourier transform in time and denote the Fourier images of $\tilde u_n(t)$ and $\tilde v_n(t)$ by $\hat u_n(\omega)$ and $\hat v_n(\omega)$ respectively. Then
\begin{equation}\label{2.5}
\(\begin{matrix} \hat u_n(\omega)\\\hat v_n(\omega)\end{matrix}\)=\(\begin{matrix} \lambda_n-\theta+i\omega&\lambda_n\\0&\lambda_n-\theta+i\omega\end{matrix}\)^{-1}\(\begin{matrix} \hat  f_n(\omega)\\\hat g_n(\omega)\end{matrix}\)
\end{equation}
and, due to the Plancherel theorem,
\begin{multline}\label{2.6}
\|\Bbb L_n\|_{\Cal L([L^2(\R)]^2,[L^2(\R)]^2)}=\\=\sup_{\omega\in\R}\bigg\|\(\begin{matrix} \lambda_n-\theta+i\omega&\lambda_n\\0&\lambda_n-\theta+i\omega\end{matrix}\)^{-1}\bigg\|_{\Cal L(\R^2,\R^2)}.
\end{multline}
Thus, the problem is actually reduced to finding the norm of $2\times2$-matrix $A_{\lambda,\theta,\omega}^{-1}$, where
\begin{equation}\label{2.7}
A_{\lambda,\theta,\omega}:=\(\begin{matrix} \lambda-\theta+i\omega&\lambda\\0&\lambda-\theta+i\omega\end{matrix}\).
\end{equation}
Moreover, as known, this norm is equal to the inverse square root of the minimal eigenvalue of the matrix
\begin{equation}\label{2.8}
A_{\lambda,\theta,\omega}A_{\lambda,\theta,\omega}^*=\(\begin{matrix} (\lambda-\theta)^2+\omega^2+\lambda^2&\lambda(\lambda-\theta-i\omega)\\
\lambda(\lambda-\theta+i\omega)&(\lambda-\theta)^2+\omega^2 \end{matrix}\).
\end{equation}
The characteristic equation reads
$$
\mu^2-(2(\lambda-\theta)^2+2\omega^2+\lambda^2)\mu+((\lambda-\theta)^2+\omega^2)^2=0
$$
and the desired minimal eigenvalue is given by
\begin{equation}\label{2.9}
\mu_{min}=\frac{2(\lambda-\theta)^2+2\omega^2+\lambda^2-\lambda\sqrt{4(\lambda-\theta)^2+
4\omega^2+\lambda^2}}2.
\end{equation}
We claim that the minimum of the function $\omega\to\mu_{min}(\lambda,\theta,\omega)$ is achieved at $\omega=0$. Indeed
$$
\partial_{\omega}\mu_{min}(\lambda,\theta,\omega)=
2\omega\(1-\frac\lambda{\sqrt{4(\lambda-\theta)^2+4\omega^2+\lambda^2}}\)
$$
and $\omega\partial_{\omega}\mu_{min}\ge0$ for all $\omega$. Thus, we  need $\mu_{min}(\lambda,\theta,\omega)$ for $\omega=0$ only and, according to \eqref{2.6}
\begin{equation}\label{2.10}
\|\Bbb L_n\|_{\Cal L([L^2(\R)]^2,[L^2(\R)]^2)}^2=\frac2{2(\lambda_n-\theta)^2+\lambda_n^2-
\lambda_n\sqrt{4(\lambda_n-\theta)^2+\lambda_n^2}}.
\end{equation}
At the next step, keeping in mind the necessity to compute the maximum of $\|\Bbb L_n\|$ with respect to $n$, we study the dependence of $\mu_{min}(\lambda,\theta,0)$ on $\lambda$. Indeed, as not difficult to check
\begin{multline}
\partial_{\lambda}\mu_{min}(\lambda,\theta,0)=-\frac{(2\theta-3\lambda)\(
\sqrt{4(\lambda-\theta)^2+\lambda^2}-\lambda\)+2(\theta-\lambda)^2}
{\sqrt{4(\lambda-\theta)^2+\lambda^2}}=\\=
\frac{-2(\theta-\lambda)^2\(4\theta-5\lambda+\sqrt{4(\lambda-\theta)^2+\lambda^2}\)}
{\sqrt{4(\lambda-\theta)^2+\lambda^2}\(\sqrt{4(\lambda-\theta)^2+\lambda^2}+\lambda\)}=\\=
\frac{8(\lambda-\theta)^3\(\theta+\sqrt{4(\lambda-\theta)^2+\lambda^2}\)}
{\sqrt{4(\lambda-\theta)^2+\lambda^2}\(\sqrt{4(\lambda-\theta)^2+\lambda^2}+\lambda\)^2}.
\end{multline}
Thus, the function $\lambda\to\mu_{min}(\lambda,\theta,0)$ is monotone decreasing for $\lambda\le\theta$ and monotone increasing for $\lambda\ge\theta$. This gives us the following result.
\begin{lemma} Let $\lambda_{n}<\theta<\lambda_{n+1}$. Then,
\begin{multline}\label{2.11}
\|\Bbb L\|^2_{\Cal L(L^2_{e^{\theta t}}(\R,\Bbb H),L^2_{e^{\theta t}}(\R,\Bbb H))}=\\=\max\left\{\mu_{min}(\lambda_n,\theta,0)^{-1},\mu_{min}(\lambda_{n+1},\theta,0)^{-1}\right\}.
\end{multline}
\end{lemma}
Our next task is to find the optimal value of $\theta\in(\lambda_n,\lambda_{n+1})$ which minimizes the norm. To this end, we note that
\begin{multline}
\partial_{\theta}\mu_{min}(\lambda,\theta,0)=
-2(\lambda-\theta)\(1-\frac{\lambda}{\sqrt{4(\lambda-\theta)^2+\lambda^2}}\)=\\=
-\frac{8(\lambda-\theta)^3}{\sqrt{4(\lambda-\theta)^2+\lambda^2}\(\sqrt{4(\lambda-\theta)^2+\lambda^2}
+\lambda\)}.
\end{multline}
Thus, the function $\theta\to\mu_{min}(\lambda,\theta,0)$ is monotone decreasing for $\theta<\lambda$ and monotone increasing for $\theta>\lambda$ and the following result is proved.

\begin{lemma} For every $n\in\Bbb N$ such that $\lambda_n<\lambda_{n+1}$, there exists a unique $\theta\in(\lambda_n,\lambda_{n+1})$ which maximizes  the norm of $\Bbb L$ and the value of $\theta$ can be found as a unique solution of the equation
\begin{equation}\label{2.12}
\mu_{min}(\lambda_{n},\theta,0)=\mu_{min}(\lambda_{n+1},\theta,0),\ \ \theta\in(\lambda_n,\lambda_{n+1}).
\end{equation}
\end{lemma}
Thus, it remains to solve equation \eqref{2.12}. To this end, we note that
$$
\mu_{min}(\lambda,\theta,0)=\(\frac{\sqrt{4(\lambda-\theta)^2+\lambda^2}-\lambda}2\)^2.
$$
Moreover, the root $\nu=\nu(\lambda,\theta):=\sqrt{\mu_{min}(\lambda,\theta,0)}$ solves the following equation
$$
\nu^2+\lambda\nu-(\lambda-\theta)^2=0.
$$
Therefore, to solve \eqref{2.12}, we need to find the common root of the following two equations
$$
\nu^2+\lambda_n\nu-(\lambda_n-\theta)^2=0,\ \ \nu^2+\lambda_{n+1}\nu-(\lambda_{n+1}-\theta)^2=0.
$$
Substructing the first equation from the second one, we get
\begin{equation}\label{2.13}
\nu=\lambda_{n+1}+\lambda_n-2\theta
\end{equation}
and inserting this value to the first equation, we arrive at the desired equation for $\theta$:
$$
3\theta^2-4(\lambda_{n+1}+\lambda_n)\theta+(\lambda_n+\lambda_{n+1})^2+\lambda_n\lambda_{n+1}=0
$$
and the only root of this equation which belongs to the required interval is
\begin{equation}\label{2.14}
\theta=\frac23(\lambda_{n+1}+\lambda_n)-\frac13\sqrt{\lambda_{n+1}^2-\lambda_n\lambda_{n+1}+\lambda_n^2}.
\end{equation}
This gives
\begin{multline*}
\nu=\frac23\sqrt{\lambda_{n+1}^2-\lambda_n\lambda_{n+1}+\lambda_n^2}-\frac13(\lambda_n+\lambda_{n+1})=\\=
\frac{(\lambda_{n+1}-\lambda_n)^2}
{\lambda_{n+1}+\lambda_n+2\sqrt{\lambda_{n+1}^2-\lambda_n\lambda_{n+1}+\lambda_n^2}}
\end{multline*}
and the proposition is proved.
\end{proof}
We now consider the analogous problem on a semi-interval $\R_-$:
\begin{equation}\label{2.half}
\Dt\xi+\Bbb A\xi=h(t),\ \  t\le0,\ \ P_n\xi(0)=\xi_0^+\in\Bbb H_n,
\end{equation}
where $P_n:H\to P_n H\sim\R^n$ is the orthoprojector to the first $n$ eigenvectors of the operator $A$:
$$
P_nu:=\sum_{i=1}^n(u,e_n)e_n
$$
and $\Bbb H_n:=P_n\Bbb H=P_n H\times P_nH\sim \R^{2n}$.
\par
To solve this problem we will use Proposition \eqref{Prop2.1}. Namely, we extend a function $h\in L^2_{e^{\theta t}}(\R_-,\Bbb H)$ by zero for positive $t$ (for simplicity, we denote this extension by $h$ again). Then, the function $\tilde \xi:=\Bbb L h$ belongs to $L^2_{e^{\theta t}}(\R,\Bbb H)$ and solves \eqref{2.half} with the appropriate initial conditions. We claim that
\begin{equation}\label{2.0}
P_n\tilde\xi(0)=0.
\end{equation}
Indeed, by definition $\tilde\xi(t)$ solves the homogeneous problem
\begin{equation}\label{2.hom}
\Dt\tilde\xi+\Bbb A\tilde\xi=0,\ t>0
\end{equation}
(since $h$ is extended by zero for positive $t$) and belongs to $L^2_{e^{\theta t}}(\R_+,\Bbb H)$. Expanding the function $\tilde \xi(t)$ in the Fourier series with respect to the base $\{e_k\}_{k=1}^\infty$, we get the equations
$$
\Dt \tilde u_k+\lambda_k \tilde u_k+\lambda_k \tilde v_k=0,\ \ \Dt \tilde v_k+\lambda_k\tilde v_k=0
$$
which can be solved explicitly:
\begin{equation}\label{2.sol}
\tilde v_k(t)=\tilde v_k(0)e^{-\lambda_k t},\ \ \tilde u_k(t)=(-\lambda_k\tilde v_k(0)t+\tilde u_k(0))e^{-\lambda_k t}.
\end{equation}
Recall that $\lambda_n<\theta<\lambda_{n+1}$. By this reason, if $k\le n$ the solutions \eqref{2.sol} can belong  to $L^2_{e^{\theta t}}(R_+,\R^2)$ only if $\tilde v_k(0)=\tilde u_k(0)=0$. This proves \eqref{2.0}. Thus, the difference $\widehat\xi(t):=\xi(t)-\tilde\xi(t)$ solves \eqref{2.half} with $h=0$. The general solution for it is again given by \eqref{2.sol}. But now we solve it backward in time, so for the component $(\widehat u_k(t),\widehat v_k(t))^t$ to belong to the space $L^2_{e^{\theta t}}(\R_-,\R^2)$, we should have $(\widehat u_k(0),\widehat v_k(0))^t=0$ for all $k>n$ and the initial data for the lower modes ($k\le n$) may be chosen arbitrarily. Let us denote by
$$
\Bbb T:\, P_n\Bbb H\to L^2_{e^{\theta t}}(\R_-,\Bbb H)
$$
the solution operator for the problem \eqref{2.half} with $h=0$ ($\xi:=\Bbb T\xi_0^+$). Then, we have proved the following result which is the main technical tool for proving the existence of inertial manifolds for the non-linear equation via the Perron method.

\begin{cor}\label{Cor2.half} Let $\lambda_n<\lambda_{n+1}$ and $\theta\in(\lambda_n,\lambda_{n+1})$
is fixed by \eqref{2.min}. Then, for every $\xi_0^+\in P_n\Bbb H$ and every $h\in L^2_{e^{\theta t}}(\R_-,\Bbb H)$, problem \eqref{2.half} possesses a unique solution $\xi\in L^2_{e^{\theta t}}(\R_-,\Bbb H)$. This solution is given by
\begin{equation}\label{2.inv}
\xi=\Bbb L h+\Bbb T\xi_0^+,
\end{equation}
where the operator $\Bbb L$ satisfies \eqref{2.15} (with $\R$ replaced by $\R_-$).
\end{cor}
We conclude this section by considering the particular case of problem \eqref{2.1} where $h=(0,g)$. This case corresponds to the particular form \eqref{0.kwa-kwak} of the non-linear equation. In this case we need to estimate only the $u$-component of the solution $\xi$, so it is natural to consider the solution operator
\begin{equation}
{\rm L}: L^2_{e^{\theta t}}(\R,H)\to L^2_{e^{\theta t}}(\R,H),\ \ {\rm L}g:=\Pi_1\Bbb L(0,g)^t,
\end{equation}
where $\Pi_1:\Bbb H\to H$ is a projection to the first component of the Cartesian product. Of course, we may estimate the norm of this operator using already obtained estimates for the solution operator $\Bbb L$, however, its special structure allows us to get better estimates.

\begin{prop}\label{Prop2.tr} Let $\lambda_n<\lambda_{n+1}$ and $\theta\in(\lambda_n,\lambda_{n+1})$ is defined in an optimal way via
\begin{equation}\label{2.gavt}
\theta=\sqrt{\lambda_n\lambda_{n+1}}.
\end{equation}
Then
\begin{equation}\label{2.gavL}
\|{\rm L}\|_{\Cal L(L^2_{e^{\theta t}}(\R,H),L^2_{e^{\theta t}}(\R,H))}=\frac1{(\sqrt{\lambda_{n+1}}-\sqrt{\lambda_n})^2}.
\end{equation}
\end{prop}
\begin{proof} Arguing as in the proof of Proposition \ref{2.1} and putting $f_n=0$ in \eqref{2.5}, we end up with
$$
\hat u_n(\omega)=-\frac{\lambda_n}{(\lambda_n-\theta+i\omega)^2}\hat g_n(\omega).
$$
Thus, the norm of the solution operator ${\rm L}_n: \tilde g_n\to\tilde u_n$ is given by
$$
\|{\rm L}_n\|_{\Cal L(L^2_{e^{\theta t}}(\R),L^2_{e^{\theta t}}(\R))}=\sup_{\omega\in\R}\frac{\lambda_n}{(\lambda_n-\theta)^2+\omega^2}=
\frac{\lambda_n}{(\lambda_n-\theta)^2}.
$$
It is not difficult to see that the function $\lambda\to\frac\lambda{(\lambda-\theta)^2}$ is increasing for $\lambda<\theta$ and decreasing for $\lambda>\theta$, so the equation for the optimal value of $\theta$ reads
$$
\frac{\sqrt{\lambda_n}}{\theta-\lambda_n}=\frac{\sqrt{\lambda_{n+1}}}{\lambda_{n+1}-\theta}
$$
which gives \eqref{2.gavt} and inserting this value of $\theta$ to the formulas for the norms of ${\rm L}_n$, we arrive at \eqref{2.gavL} and finish the proof of the proposition.
\end{proof}
The next result is the analogue of Corollary \ref{Cor2.half} for this case and is an immediate corollary of Proposition \ref{Prop2.tr}.

\begin{cor}\label{Cor2.trhalf} Let $\lambda_n<\lambda_{n+1}$ and $\theta\in(\lambda_n,\lambda_{n+1})$
is fixed by \eqref{2.gavt}. Then, for every $\xi_0^+\in P_n\Bbb H$ and every $g\in L^2_{e^{\theta t}}(\R_-,H)$, problem \eqref{2.half} with $h:=(0,g)^t$ possesses a unique solution $\xi\in L^2_{e^{\theta t}}(\R_-,\Bbb H)$. The $u$-component of this solution is given by
\begin{equation}
u={\rm L} g+{\rm T}\xi_0^+,
\end{equation}
where the operator $\Bbb L$ satisfies \eqref{2.gavL} (with $\R$ replaced by $\R_-$) and ${\rm T}:=\Pi_1\Bbb T$.
\end{cor}
\begin{rem} Since
$$
\sqrt{\lambda_n^2-\lambda_n\lambda_{n+1}+\lambda_{n+1}^2}\ge\sqrt{\lambda_n\lambda_{n+1}}
$$
and the equality is possible only if $\lambda_n=\lambda_{n+1}$, the truncated estimate \eqref{2.gavL} is indeed better than the analogous estimate \eqref{2.15} for the full system.
\end{rem}

\section{Inertial manifolds and spectral gap conditions}\label{s3}
In this section we give the spectral gap conditions for existence of inertial manifolds for semilinear parabolic equation of the form
\begin{equation}\label{3.kwak}
\Dt\xi+\Bbb A\xi=\Bbb F(\xi),\ \ \xi\big|_{t=0}=\xi_0, \ \ \xi=(u,v)^t\in\Bbb H,
\end{equation}
where $\Bbb H:=H\times H$ is a Cartesian square of an abstract Hilbert space $H$. As before, the operator $\Bbb A$ is assumed to have the following structure:
\begin{equation}
\Bbb A=\(\begin{matrix}1&1\\0&1\end{matrix}\)A,
\end{equation}
where $A: D(A)\to H$ is a positive {\it self-adjoint} linear operator
 in $H$ with a compact inverse. As we have seen in examples of Section \ref{s1}, this form of equations is typical for the Kwak transform.
 \par
 We assume that the cut-off procedure is already performed outside of the global attractor and, therefore, the non-linearity $\Bbb F$ is globally Lipschitz in $\Bbb H$ with a Lipschitz constant $L$:
 \begin{equation}\label{3.lip}
 \|\Bbb F(\xi_1)-\Bbb F(\xi_2)\|_{\Bbb H}\le L\|\xi_1-\xi_2\|_{\Bbb H},\ \ \xi_i\in\Bbb H,\ \ i=1,2.
 \end{equation}
 As usual, we denote by $\{\lambda_k\}_{k=1}^\infty$ and $\{e_k\}_{k=1}^\infty$ the eigenvalues of $A$ enumerated in the non-decreasing order and the corresponding eigenvectors respectively. The orthoprojector to the plane $P_nH$ generated by the first $n$ eigenvectors of $A$ we denote by $P_n$ and $Q_n:={\rm Id}-P_n$. We also use the notation
 $$
 P_n\xi=P_n(u,v)^t:=(P_nu,P_nv)^t\in P_nH\times P_n H:=P_n\Bbb H
 $$
 and  analogously for the projector $Q_n$.
\par
We start with recalling the definition of an inertial manifold (IM) adapted to the case of equation \eqref{3.kwak}, see \cite{19} and references therein for more details.
\par
\begin{Def}\label{Def3.IM} A closed finite-dimensional submanifold $\Bbb M$ of the phase space $\Bbb H$ is an IM for equation \eqref{3.kwak} with the base $P_n\Bbb H$ if
\par
1) It is strictly invariant with respect to the solution semigroup $\Bbb S(t):\Bbb H\to\Bbb H$ generated by equation \eqref{3.kwak}: $\Bbb S(t)\Bbb M=\Bbb M$ for all $t\ge0$;
\par
2) It is a graph of a globally Lipschitz function $M:P_n\Bbb H\to Q_n\Bbb H$, i.e.
$$
\Bbb M=\{\xi^+_0+M(\xi^+_0),\ \ \xi_0^+\in\Bbb P_n\Bbb H\},
$$
in particular, it is homeomorphic to $\R^{2n}$;
\par
3) It possesses the so-called exponential tracking (asymptotic phase) property. Namely, there exists $\theta>0$ and a monotone function $Q$ such that, for every solution $\xi(t)$, $t\in\R_+$, of equation \eqref{3.kwak}, there is a trace solution $\bar\xi(t)$ of \eqref{3.kwak} belonging to $\Bbb M$ for all $t\ge0$ such that
\begin{equation}\label{3.track}
\|\xi(t)-\bar\xi(t)\|_{\Bbb H}\le Q(\|\xi(0)\|_{\Bbb H})e^{-\theta t}
\end{equation}
for all $t\ge0$.
\end{Def}
The next theorem which gives the conditions for the existence of IM for problem \eqref{3.kwak} can be considered as a main result of this section.

\begin{theorem}\label{Th3.main} Let $n\in\Bbb N$ be such that the spectral gap condition
\begin{equation}\label{3.gap}
\frac{(\lambda_{n+1}-\lambda_n)^2}
{\lambda_{n+1}+\lambda_n+2\sqrt{\lambda_{n+1}^2-\lambda_n\lambda_{n+1}+\lambda_n^2}}>L
\end{equation}
is satisfied. Then, equation \eqref{3.kwak} possesses $2n$-dimensional IM with the base $P_n\Bbb H$. Moreover, the exponent $\theta$ in the tracking property satisfies \eqref{2.min}.
\end{theorem}
\begin{proof} Similarly to \cite{ACZ,17,19} (see also references therein), we follow the so-called Perron method for constructing the IM. According to this method, for every $\xi_0^+\in P_n\Bbb H$, we need to find a unique backward in time solution of the equation
\begin{equation}\label{3.au}
\Dt\xi+\Bbb A\xi=\Bbb F(\xi),\ \ P_n\xi\big|_{t=0}=\xi_0^+,\ \ t\le0
\end{equation}
and then define $M(\xi_0^+):=Q_n\xi(0)$.
\par
Equation \eqref{3.au} can be easily solved using Corollary \ref{Cor2.half} and Banach contraction theorem. Indeed, due to \eqref{2.inv}, equation \eqref{3.au} is equivalent to the following fixed point problem:
\begin{equation}\label{3.fix}
\xi=\Bbb L\Bbb F(\xi)+\Bbb T\xi_0^+
\end{equation}
in the space $L^2_{e^{\theta t}}(\R_-,\Bbb H)$. Due to estimate \eqref{2.15} for the norm of the operator $\Bbb L$ and assumption \eqref{3.gap}, we see that the right-hand side of \eqref{3.fix} is a contraction with respect to $\xi\in L^2_{e^{\theta t}}(\R_-,\Bbb H)$. Since $\Bbb T$ is a bounded linear operator, by the Banach contraction theorem, \eqref{3.fix} is indeed uniquely solvable and the solution map $\mathcal M:P_n\Bbb H\to \xi\in L^2_{e^{\theta t}}(\R_-,\Bbb H)$, $\mathcal M:\xi_0^+\to\xi$, is globally Lipschitz. Furthermore, due to parabolic smoothing property, it is not difficult to check that the map $\mathcal M$ is Lipschitz also as a map
$$
\mathcal M:P_n\Bbb H\to W^{1,2}_{e^{\theta t}}(\R_-,\Bbb H)\subset C_{e^{\theta t}}(\R_-,\Bbb H),
$$
see e.g., \cite{19} for the details. Therefore, the map $M:\xi_0^+\to Q_n\mathcal M(\xi_0^+)\big|_{t=0}$ is also well-defined and Lipschitz continuous.
\par
Thus, the Lipschitz continuous submanifold $\Bbb M$ of $\Bbb H$ with the base $P_n\Bbb H$ is constructed. Its invariance follows immediately from the construction and from the uniqueness part of the Banach contraction theorem. So, we only need to check the exponential tracking property. This is also a standard and straightforward corollary of estimate \eqref{2.15} for the solution operator for the linear equation, so we only give a sketch of the proof leaving the details to the reader, see also \cite{19}.
\par
We assume for simplicity that $\Bbb F(0)=0$. Let $\xi(t)$, $t\ge0$, be an arbitrary trajectory of equation \eqref{3.kwak} and let $\phi\in C^\infty(\R)$ be a cut-off function such that $\phi(t)\equiv0$ for $t\le0$ and $\phi(t)\equiv1$ for $t\ge1$. We seek for the desired solution $\bar\xi(t)$, $t\in\R$ belonging to the manifold $\Bbb M$ in the form
\begin{equation}\label{3.split}
\bar\xi(t)=\phi(t)\xi(t)+\tilde\xi(t),\ \ \tilde\xi\in L^2_{e^{\theta t}}(\R,\Bbb H).
\end{equation}
Indeed, since $\bar\xi(t)=\tilde\xi(t)$ for $t\le0$, we have $\bar\xi\in L^2_{e^{\theta t}}(\R_-,\Bbb H)$ and by this reason the solution $\bar\xi(t)\in\Bbb M$ for all $t\in\R$. On the other hand, for $t\ge1$ we have $\bar\xi(t)-\xi(t)=\tilde\xi(t)$ and therefore, $\bar\xi-\tilde\xi\in L^2_{e^{\theta t}}(\R_+,\Bbb H)$ and using again the parabolic smoothing property, we get \eqref{3.track}.
\par
Thus, it only remains to construct a solution $\bar\xi(t)$ of the form \eqref{3.split}. To this end, we write down the equation for the function $\tilde\xi$:
$$
\Dt\tilde\xi+\Bbb A\tilde\xi=\Bbb F(\phi\xi+\tilde\xi)-\phi\Bbb F(\xi)+\phi'\xi:=\Phi(\tilde\xi).
$$
Since $\Bbb F$ is Lipschitz with Lipschitz constant $L$, one can verify that $\Phi$ is globally Lipschitz with the same Lipschitz constant $L$ as a map from $L^2_{e^{\theta t}}(\R,\Bbb H)$ to itself. Using Proposition \ref{Prop2.1}, we rewrite this equation as a fixed point problem:
$$
\tilde\xi=\Bbb L\Phi(\tilde\xi)
$$
and the spectral gap condition now gives that the right-hand side of this equation is a contraction. Thus, the existence of $\tilde\xi$ is verified due to the Banach contraction theorem and the exponential tracking property is proved. This finishes the proof of the theorem.
\end{proof}
\begin{rem}\label{Rem.xa-xa} As not difficult to see,
\begin{equation*}
(\sqrt{\lambda_n}+\sqrt{\lambda_{n+1}})^2<\lambda_n+\lambda_{n+1}+2\sqrt{\lambda_n^2+\lambda_{n+1}^2-\lambda_n\lambda_{n+1}}<
3(\sqrt{\lambda_n}+\sqrt{\lambda_{n+1}})^2
\end{equation*}
and, therefore, we may write a sufficient condition for \eqref{3.gap} to be satisfied:
\begin{equation}\label{3.sim-gap}
\sqrt{\lambda_{n+1}}-\sqrt{\lambda_n}>\sqrt{3L}.
\end{equation}
On the other hand, as we will see below, the spectral gap condition \eqref{3.gap} is sharp, so the IM may not exist if
$$
\sqrt{\lambda_{n+1}}-\sqrt{\lambda_n}<\sqrt{L}.
$$
These conditions are very far from the standard spectral gap conditions for the case when $\Bbb A$ is {\it self-adjoint} and the non-linearity $\Bbb F$ is Lipschitz from $\Bbb H$ to $\Bbb H$. We recall that in that case the analogous condition reads
$$
\lambda_{n+1}-\lambda_n>2L.
$$
However, if we compare it with the self-adjoint case where $\Bbb F$ "eats" smoothness, namely, $\Bbb F$ is Lipschitz as a map from $D(\Bbb A^{1/2})$ to $\Bbb H$, we see a strong similarity. Indeed, in this case the sharp spectral gap condition reads
\begin{equation}\label{3.eat-gap}
\frac{\lambda_{n+1}-\lambda_n}{\lambda_n^{1/2}+\lambda_{n+1}^{1/2}}=\sqrt{\lambda_{n+1}}-\sqrt{\lambda_n}>L
\end{equation}
which coincides with \eqref{3.sim-gap} up to the change of the Lipschitz constant.
\par
Thus, starting with the equation with $F:D(A^{1/2})\to H$ and self-adjoint linear part $A$ and performing the Kwak type transform, we end up with a new equation where the non-linearity $\Bbb F$ is globally Lipschitz from $\Bbb H$ to $\Bbb H$ (does not "eat" smoothness), but with the {\it non-self-adjoint} leading part $\Bbb A$ which contains Jordan cells. As we see, this new equation requires much stronger spectral gap conditions than in the self-adjoint case which have the same structure as the conditions for the initial equation (before Kwak transform).
\par
 This explains why the Kwak transform is not helpful (at least in a straightforward way) for constructing the IMs (although as we mentioned in the introduction one still may expect some cleverly constructed Kwak transform may work, say,  due to a drastic decreasing of the Lipschitz constant).
We also mention that the crucial error in the mentioned above papers \cite{Kwak1,Kwak,Rom,Tem} on IMs via the Kwak transform is {\it exactly} the implicit assuming that the spectral gap conditions for the non-self-adjoint operator $\Bbb A$ with Jordan cell are similar to the self-adjoint case.
\end{rem}
We now discuss the sharpness of the obtained spectral gap condition \eqref{3.gap}. As usual, the absence of an IM over the base $P_n\Bbb H$ for a {\it fixed} value $n\in\Bbb N$ for which the spectral gap conditions are violated can be shown for the properly chosen {\it linear} map $\Bbb F$. Namely, the following result holds.

\begin{prop}\label{Prop3.sharp} Let the eigenvalues $\lambda_n$ and $\lambda_{n+1}$ and the Lipschitz constant $L$ be such that the spectral gap condition \eqref{3.gap} is strictly violated, i.e.,
\begin{equation}\label{3.nogap}
\frac{(\lambda_{n+1}-\lambda_n)^2}
{\lambda_{n+1}+\lambda_n+2\sqrt{\lambda_{n+1}^2-\lambda_n\lambda_{n+1}+\lambda_n^2}}<L.
\end{equation}
Then there exists a linear operator $\Bbb F:\Bbb H\to\Bbb H$ whose norm does not exceed $L$ such that the equation \eqref{3.kwak} does not possess an IM over the base $P_n\Bbb H$.
\end{prop}
\begin{proof} We recall that the operator $\Bbb A$ is block diagonal in the Fourier base $\{e_k\}_{k=1}^\infty$ and $2\times2$ matrix which corresponds to the $k$th block reads
$$
\Bbb A_k=\(\begin{matrix} \lambda_k&\lambda_k\\0&\lambda_k\end{matrix}\).
$$
Define the linear operator $\bar{\Bbb F}$ by the following formula:
\begin{equation}\label{3.fig1}
\bar{\Bbb F}_k=-\(\begin{matrix} 0&K\\K&0\end{matrix}\), \ \ k=n,n+1,\ \ \bar{\Bbb F}_k=0,\ \ k\ne n,n+1,
\end{equation}
where $K$ is a parameter which will be fixed later. Then the operator $\Bbb A-\bar{\Bbb F}$ remains block diagonal and only the $n$-th and $(n+1)$-th blocks are affected by the perturbation $\bar{\Bbb F}$. The new eigenvalues in these blocks can be easily calculated:
$$
\mu_k^-:=\lambda_k-\sqrt{K(\lambda_k+K)},\ \ \mu_k^+:=\lambda_k+\sqrt{K(\lambda_k+K)}, \ \ k=n,n+1.
$$
Let us also denote by $\vec e_n^\pm$ and $\vec e_{n+1}^\pm$ the corresponding eigenvalues.
\par
The idea of our construction is to couple the $n$th and $(n+1)$th blocks. To this end, we fix the parameter $K$ as a solution of the following equation:
$$
\mu_n^+= \lambda_n+\sqrt{K(\lambda_n+K)}=\lambda_{n+1}-\sqrt{K(\lambda_{n+1}+K)}=\mu_{n+1}^-.
$$
Solving this equation with respect to $K$ in a straightforward way, we end up with
$$
K=\frac{(\lambda_{n+1}-\lambda_n)^2}
{\lambda_{n+1}+\lambda_n+2\sqrt{\lambda_{n+1}^2-\lambda_n\lambda_{n+1}+\lambda_n^2}}.
$$
Thus, due to \eqref{3.nogap} condition, the norm of the constructed operator is strictly less than $L$. The operator $\Bbb A-\bar{\Bbb F}$ is still block diagonal, but now it possesses two equal eigenvalues $\mu_n^+$ and $\mu_{n+1}^-$ with eigenvectors $\vec e_n^+$ and $\vec e_{n+1}^-$ belonging to different blocks. Finally, to couple these blocks, we add one more (arbitrarily small) perturbation $\widetilde{\Bbb F}$ which acts only in the plane $\operatorname{span}\{\vec e^+_n,\vec e_{n+1}^-\}$ (and by this reason does not change any other eigenvalues except of $\mu_n^+$ and $\mu_{n+1}^-$), but in this plane it generates a rotation, so the perturbed eigenvalues $\mu_n^+$ and $\mu_{n+1}^-$ become complex conjugate with non-zero imaginary part.
\par
Finally, we denote $\Bbb F:=\bar{\Bbb F}+\widetilde{\Bbb F}$. We claim that $\Bbb F$ is a desired linear operator. Indeed, by construction, the norm of $\bar{\Bbb F}$ is strictly less than $L$ and the perturbation $\widetilde{\Bbb F}$ can be chosen arbitrarily small, so the norm of $\Bbb F$ is less than $L$ as well.
\par
So, it only remains to show that equation \eqref{3.kwak} with this choice of $\Bbb F$ does not possess an IM with the base $P_n\Bbb H$. Assume that such a manifold $\mathbb M$ exists. Then the projector $P_n:\Bbb M\to P_n\Bbb H$ must be one-to-one. Let us consider one dimensional plane $\Cal H_1:=\R\vec e_n^+\subset P_n\Bbb H$ and its image $\mathcal M_1=P_n^{-1}\mathcal H_1\subset\Bbb M$ on the manifold. Clearly $\mathcal M_1$ is invariant with respect to the time evolution generated by equation \eqref{3.kwak}. On the other hand, if we denote by $x(t)$ and $y(t)$ the components of a solution of \eqref{3.kwak} which correspond to the vectors $\vec e_n^+$ and $\vec e_{n+1}^-$ respectively, we get an explicit formula
\begin{multline*}
x(t)=e^{-\mu t}\(x(0)\cos(\omega t)+y(0)\sin(\omega t)\),\\ y(t)=e^{-\mu t}\(y(0)\cos(\omega t)-x(0)\sin(\omega t)\),
\end{multline*}
where $\mu=\operatorname{Re}\mu_n^+$ and $\omega=\operatorname{Im}\mu_n^+$. Since by the construction $\omega\ne0$, $x(t)$ oscillates and, in particular, has infinitely many zeros. Since the considered trajectory is not periodic, this contradicts to the injectivity of $P_n$ at zero. Thus, the inertial manifold cannot exist and the proposition is proved.
\end{proof}
\begin{rem} The proved proposition demonstrate the absence of an IM for dimension $2n$ only and with the base $P_n\Bbb H$ only and do not exclude its existence for different dimensions or/and different bases. However, if we assume that the spectral gap conditions \eqref{3.gap} are not satisfied for any $n\in\Bbb N$, namely, that the condition
\begin{equation}\label{3.nogaps}
\sup_{n\in\Bbb N}\left\{\frac{(\lambda_{n+1}-\lambda_n)^2}
{\lambda_{n+1}+\lambda_n+2\sqrt{\lambda_{n+1}^2-\lambda_n\lambda_{n+1}+\lambda_n^2}}\right\}<L
\end{equation}
holds, then following the scheme suggested in \cite{5}(see also \cite{AZ2,19}) one can construct a nonlinearity $\Bbb F$ which is globally bounded and Lipchitz with the constant $L$ as a map from $\Bbb H$ to $\Bbb H$, such that the global attractor of equation \eqref{3.kwak} does not belong to any finite dimensional Lipschitz submanifold of the phase space $\Bbb H$. Moreover, the dynamics generated by \eqref{3.kwak} on this attractor is infinite dimensional. In particular, there are two distinct trajectories $\xi_1(t)$ and $\xi_2(t)$ belonging to the attractor such that
$$
\|\xi_1(t)-\xi_2(t)\|_{\Bbb H}\le Ce^{-\alpha t^3}, \ \alpha>0,\ \ t>0,
$$
see \cite{AZ2,19} for more details. The proof of these results is rather technical although follows word by word the construction given in \cite{5,19} (with the proper minor corrections related with the concrete structure of equation \eqref{3.kwak}). In order to avoid the technicalities we will not give the rigorous proof here.
\end{rem}
To conclude the section, we briefly consider the following particular case of equation \eqref{3.kwak}:
\begin{equation}\label{3.wave}
\Dt\(\begin{matrix} u\\v\end{matrix}\)+\(\begin{matrix} 1&1\\0&1\end{matrix}\)A\(\begin{matrix} u\\v\end{matrix}\)=\(\begin{matrix} 0\\F(u)\end{matrix}\),
\end{equation}
where the nonlinearity $\Bbb F(u,v):=(0,F(u))^t$. As we have discussed in Section \ref{s1}, this particular form is typical for some versions of the Kwak transform. Of course, we may treat this equation as \eqref{3.kwak} and use the spectral gap condition \eqref{3.gap}, but this is not optimal since the specific form of $\Bbb F$ allows us to apply the Banach contraction theorem in the functional space which do not contain the $v$-component and this makes the spectral gap conditions better. Namely, the following result holds.

\begin{theorem}\label{Th3.main2}Let $n\in\Bbb N$ be such that the spectral gap condition
\begin{equation}\label{3.gapbest}
 \sqrt{\lambda_{n+1}}-\sqrt{\lambda_n}>\sqrt L
\end{equation}
is satisfied where $L$ is a global Lipchitz constant for the map $F:H\to H$. Then, equation \eqref{3.wave} possesses $2n$-dimensional IM with the base $P_n\Bbb H$. Moreover, the exponent in the tracking property satisfies $\theta=\sqrt{\lambda_n\lambda_{n+1}}$. Moreover, the spectral gap condition \eqref{3.gapbest} is sharp in the sense that the analogue of Proposition \ref{Prop3.sharp} also holds.
\end{theorem}
\begin{proof}[Sketch of proof] Analogously to the proof of Theorem \ref{Th3.main}, we seek the desired manifold via the  solutions of the backward problem \eqref{3.au} belonging to the functional space $L^2_{e^{\theta t}}(\R_-,\Bbb H)$ where the exponent $\theta$ should be chosen in an optimal way. However, in our special case, we need only the norm of the $u$-component of the solution $\xi(t)=(u(t),v(t))^t$ in order to control the nonlinearity, so we need to optimize only the norm of the $u$-component. After the desired $u$-component of the solution will be constructed via the implicit function theorem, the $v$-component can be easily restored from the second equation of \eqref{3.wave}. To be more precise, we use Corollary \ref{Cor2.trhalf} and rewrite the auxiliary problem \ref{3.au} in the form
\begin{equation}\label{3.aa}
u={\rm L}F(u)+{\rm T}\xi_0^+,\ \ u\in L^2_{e^{\theta t}}(\R_-,H).
\end{equation}
Then, assumption \eqref{3.gapbest} together with equality \eqref{2.gavL} guarantee that the right-hand side of \eqref{3.aa} is a contraction and, therefore, the desired solution $\xi=(u,v)^t$ for the auxiliary problem \eqref{3.au} exists. The proof of the existence of the IM is completed exactly as in the proof of Theorem \ref{Th3.main}.
\par
The sharpness of the spectral gap condition \eqref{3.gapbest} can be verified similarly to the proof of Proposition \ref{Prop3.sharp}. However, since we now have less freedom in the choice of operators $\bar{\Bbb F}$ and $\widetilde{\Bbb F}$, some extra accuracy is required. Namely, the analogue of formula \eqref{3.fig1} now reads
\begin{equation}\label{3.fig2}
\bar{\Bbb F}_k=-\(\begin{matrix} 0&0\\K&0\end{matrix}\), \ \ k=n,n+1,\ \ \bar{\Bbb F}_k=0,\ \ k\ne n,n+1.
\end{equation}
Then, for the parameter $K$, we get the equation
$$
\mu_n^+=\lambda_n+\sqrt{K\lambda_n}=\lambda_{n+1}-\sqrt{K\lambda_{n+1}}=\mu_{n+1}^-
$$
which gives
$$
K=(\sqrt{\lambda_{n+1}}-\sqrt{\lambda_n})^2,\ \ \mu_n^+=\mu_{n+1}^-=\sqrt{\lambda_{n+1}\lambda_n}.
$$
The construction of the operator $\widetilde{\Bbb F}$ is a bit more complicated since we are not able to perturb the equation arbitrarily. Namely, to preserve the structure of the equation, we consider the following 4 dimensional perturbation in the base related with $\{u_{n},v_n,u_{n+1},v_{n+1}\}$:
\begin{equation}
\widetilde {\Bbb F}=-\(\begin{matrix}0&0&0&0\\0&0&\eb&0\\0&0&0&0\\\eb&0&0&0 \end{matrix}\).
\end{equation}
Then computing the determinant
$$
f(\lambda):=\operatorname{det}\(\begin{matrix}\lambda_n-\lambda&\lambda_n&0&0\\
K&\lambda_n-\lambda&\eb&0\\0&0&\lambda_{n+1}-\lambda&\lambda_{n+1}\\
\eb&0&K&\lambda_{n+1}-\lambda \end{matrix}\)
$$
and putting $\lambda=y+\sqrt{\lambda_n\lambda_{n+1}}$ and $K=(\sqrt{\lambda_{n+1}}-\sqrt{\lambda_n})^2$, we arrive at
\begin{multline}\label{3.maple}
f(y+\sqrt{\lambda_n\lambda_{n+1}})=-\eb^2\lambda_n\lambda_{n+1}-\\-4\sqrt{\lambda_n\lambda_{n+1}}
(\sqrt{\lambda_{n+1}}-\sqrt{\lambda_n})^2y^2-2(\sqrt{\lambda_{n+1}}-\sqrt{\lambda_n})^2y^3+y^4.
\end{multline}
The last formula shows that for small $\eb$ the eigenvalues $\mu_n^+(\eb)$ and $\mu_{n+1}^-(\eb)$ which corresponds to $\mu_n^+=\mu_{n+1}^-$ become complex conjugate.
\par
Note also that, in contrast to the proof of Proposition \ref{Prop3.sharp}, the corresponding eigenvectors $\vec e_{n}^\pm(\eb)$ and $\vec e_{n+1}^\pm(\eb)$ depend explicitly on $\eb$. Since this dependence is continuous, then the existence of Lipschitz IM over the base $P_n\Bbb H$ implies the existence of an IM over
$$
P_{n-1}\Bbb H\times \operatorname{span}\{\vec e_n^-(\eb),\vec e_{n}^+(\eb)\}
$$
if $\eb$ is small enough (due to the fact that the set of bi-Lipschitz projectors is open, see \cite{19}). The rest of the proof is exactly the same as in Proposition \ref{Prop3.sharp}. Thus, the theorem is proved.
\end{proof}
\begin{rem}\label{Rem3.last} The presence of the square root of the Lipschitz constant $L$ may look unnatural. The nature of this root can be clarified if we use an alternative method of constructing the IM for \eqref{3.wave} by reducing the equation to the self-adjoint case as described in Remark \ref{Rem1.bla}, namely, to equation \eqref{1.saa}. In order to see where the square root comes from, we just need to consider this equation in the space $\Bbb H:=H\times H$ endowed by an optimal norm:
$$
\|\xi\|_{\Bbb H_L}^2:=L\|u\|^2_H+\|v\|^2_H.
$$
Indeed, in this special metric, we have
\begin{multline}
\|\Bbb F(\tilde u_1,v_1)-\Bbb F(\tilde u_2,v_2)\|_{\Bbb H_L}^2=\\=L\|A^{1/2}(v_1-v_2)\|^2_H+\|F(A^{1/2}\tilde u_1)-F(A^{1/2}\tilde u_2)\|^2_H\le
 \\\le L(L\|A^{1/2}(\tilde u_1-\tilde u_2)\|^2_H+\|A^{1/2}(v_1-v_2)\|^2_H)=L\|A^{1/2}(\xi_1-\xi_2)\|^2_{\Bbb H_L}.
\end{multline}
Thus, the Lipschitz constant for the map $\Bbb F$ in this case is exactly $\sqrt L$. Moreover, applying the standard spectral gap conditions to this self-adjoint case we see that the IM exists at least if
$$
\frac{\lambda_{n+1}-\lambda_n}{\lambda_n^{1/2}+\lambda_{n+1}^{1/2}}=
\sqrt{\lambda_{n+1}}-\sqrt{\lambda_n}>\sqrt L
$$
which coincides with assumption \eqref{3.gapbest} of Theorem \ref{Th3.main}.
\end{rem}

\end{document}